\documentclass[11pt]{amsart}
\usepackage{amsmath}
\usepackage{amsthm}
\usepackage{amssymb}
\usepackage{amsfonts,mathrsfs}
\usepackage{stmaryrd}
\usepackage{amsxtra}  
\usepackage{epsfig}
\usepackage{verbatim}
\usepackage[all]{xy}

\theoremstyle{plain}
\newtheorem{theorem}[equation]{Theorem}

\newtheorem{lemma}[equation]{Lemma}
\newtheorem{corollary}[equation]{Corollary}

\theoremstyle{remark}
\newtheorem{remark}[equation]{Remark}
\numberwithin{equation}{section}

\def\norm#1{\Vert#1\Vert}

\newcommand{\ca}{{\mathcal A}}

\newcommand{\ci}{{\mathcal I}}

\newcommand{\ck}{{\mathcal K}}

\newcommand{\co}{{\mathcal O}}

\newcommand{\B}{{\mathbb B}}
\newcommand{\C}{{\mathbb C}}

\newcommand{\R}{{\mathbb R}}

\newcommand{\Z}{{\mathbb Z}}

\begin{document}

\title[$L^p$ Bergman]{The Bergman projection on fat Hartogs triangles: $L^p$ boundedness}
\author{L. D. Edholm  \& J. D. McNeal}
\subjclass[2010]{32W05}
\begin{abstract}
 A class of pseudoconvex domains in $\mathbb{C}^{n}$ generalizing the Hartogs triangle is considered. The $L^p$ boundedness of the Bergman projection
 associated to these domains is established, for a restricted range of $p$ depending on the ``fatness'' of domains. This range of $p$ is shown to be sharp.
 \end{abstract}
\thanks{Research of the second author was partially supported by a National Science Foundation grant.}
\address{Department of Mathematics, \newline The Ohio State University, Columbus, Ohio, USA}
\email{edholm@math.ohio-state.edu}
\address{Department of Mathematics, \newline The Ohio State University, Columbus, Ohio, USA}
\email{mcneal@math.ohio-state.edu}

\maketitle 


\section{Introduction}\label{S:intro}

For $k \in \Z^+$, define the domain $\Omega_k\subset\C^2$ by

\begin{equation}\label{D:fatHartogs}
\Omega_k := \{(z_1, z_2) \in \C^2: |z_1|^k < |z_2| < 1 \},
\end{equation}
and call $\Omega_k$ the fat Hartogs triangle of exponent $k$. When $k=1$, \eqref{D:fatHartogs} is the classical Hartogs triangle, a well-known
example of a pseudoconvex domain with non-trivial Nebenh\" ulle. As $k\to\infty$, $\Omega_k$ becomes ``fatter'', filling out the product domain
$D\times D^*$ in the limit, where $D= \left\{z:|z|< 1\right\}$ is the unit disc and  $D^*=\left\{z: 0<|z|< 1\right\}$ the punctured unit disc in $\C$.

The purpose of this paper is to establish how the Bergman projection acts on the Lebesgue spaces $L^p\left(\Omega_k\right)$. If $A^p\left(\Omega_k\right)\subset L^p\left(\Omega_k\right)$ denotes the
closed subspace of holomorphic functions, let $\mathbf{B}_k: L^2(\Omega_k) \to A^2(\Omega_k)$  denote the Bergman projection (orthogonal projection) associated to $\Omega_k$. Our main result is:

\begin{theorem}\label{T:main}
The Bergman projection $\mathbf{B}_k$ is a bounded operator from $L^p(\Omega_k)$ to $A^p(\Omega_k)$ if and only if $p \in (\frac{2k+2}{k+2}, \frac{2k+2}{k})$.

\end{theorem}

The most notable aspect of our result comes by comparison with the known behavior of the Bergman projection associated to smoothly bounded, pseudoconvex domains. There is no general
result for all such domains, but $L^p$ mapping properties of the Bergman projection have been determined for large classes of domains, a subset of the family of {\it finite type}
domains defined in \cite{DAngelo82}. See \cite{PhoSte77}, \cite{McNeal89}, \cite{McNSte94}, \cite{NagRosSteWai89}, \cite{Koenig02} for the principal results. In all these cases, the Bergman projection is bounded on $L^p$ for all $1 <p<\infty$. In two more recent works,
\cite{EhsLie08} and \cite{LanSte12}, the Bergman projection is also shown to be bounded for the full range $1 <p<\infty$ on some domains with less smooth boundary, but only for strongly pseudoconvex domains in this class.

Our domains $\Omega_k$ are bounded and pseudoconvex. But in contrast to the situations above, the boundary of $\Omega_k$, $b\Omega_k$, has a serious singularity at the origin. 
Indeed, $b\Omega_k$ cannot be represented as the graph of a continuous function in a neighborhood of $(0,0)$. This boundary singularity causes the limited range of $L^p$
boundedness in Theorem \ref{T:main}, though our present understanding of this matter is only phenomenological. Determining how the type of boundary singularity of a domain limits 
$L^p$ mapping behavior of its Bergman projection remains an intriguing open problem, and is highly connected to many other questions in complex analysis. We do point out that
our result shows a somewhat surprising degeneracy of an analytic property relative to limits of domains. Namely, note that the range of $L^p$ boundedness for $\mathbf{B}_k$ in Theorem \ref{T:main} 
shrinks to 2 as $k\to\infty$, even though the domains $\Omega_k$ tend set-theoretically to the domain $D\times D^*$, whose Bergman projection is $L^p$ bounded for all $1 <p<\infty$.

Several interesting, recent papers overlap with our work. Zeytuncu, \cite{Zey13}, constructed bounded domains in $\C^n$ whose Bergman projection is only bounded on $L^p$ for a restricted
range of $p$, namely a H\"older symmetric interval of values centered at 2. He also constructed domains whose Bergman projection is only bounded on $L^p$ for $p=2$.
In \cite{ChaZey14}, Chakrabarti and Zeytuncu consider the classical Hartogs triangle, $\Omega_1$ above, and prove $L^p$ boundedness for the Bergman projection
corresponding to Theorem \ref{T:main} for the case $k=1$. They also observed a new phenomena on $\mathbf{B}_1$ related to its weighted $L^p$ mapping behavior. In \cite{Chen14}, Chen generalizes the Hartogs triangle differently than \eqref{D:fatHartogs}, essentially considering domains of the form $\left\{(z_1, z_2) \in \Omega \times \C: \left|\phi\left(z_1\right)\right| < |z_2| < 1 \right\}$ for $\phi: \Omega \to D$ a biholomorphic map
onto $D\subset\C$. He then obtains $L^p$ boundedness of the Bergman projection for the same restricted range of $p$ as \cite{ChaZey14} shows for the Hartogs triangle.  Earlier examples of restricted $L^p$ mapping behavior of the Bergman projection on specific domains were given by Lanzani and Stein, \cite{LanSte04}, and Krantz and Peloso, \cite{KraPel08}.

The situation considered in this paper differs from that in \cite{ChaZey14} and \cite{Chen14} in one important respect: the domains $\Omega_k$ are not biholomorphic
to $D\times D^*$ if $k>1$. Indeed, the first author shows in \cite{Edh15} that the Bergman kernel associated to $\Omega_k$ has zeroes if $k>1$, while the Bergman kernel
of $D\times D^*$ is known to be zero-free. Thus, to prove Theorem \ref{T:main} we must directly address the Bergman projection on $\Omega_k$, rather than transfer to $D\times D^*$ as done in \cite{ChaZey14} and \cite{Chen14}. 

There is a general, though technical, point in our proofs below that seems worth noting. Once the crucial estimate on the Bergman kernel of $\Omega_k$, \eqref{E:basicEstimate}, is in hand, we focus on suitably decomposing $\Omega_k$ in order to get the desired integral estimates. This basic method is derived from \cite{McNSte94}, and is used particularly in Lemma \ref{L:calculus1}
and Lemma \ref{L:calculus2}. The method displays the interplay between the singularity of the Bergman kernel and the shape of $\Omega_k$ most clearly, thus showing how to generalize the mapping result to other kernels or to the Bergman kernel on domains other than $\Omega_k$. For an example of the former, see Corollary \ref{C:extendedLp}. The semi-classical proofs of the $L^p$ boundedness of $\mathbf{B}_\Omega$ -- for $\Omega$ the disc in $\C$, the ball in $\C^n$, or the polydisc in $\C^n$, see \cite{RudinFunctiontheory}, \cite{ForRud74}, \cite{Chen14} -- are based on power series computations and the asymptotics of the Gamma or Beta functions, which are less amenable for generalization to domains without circular symmetry.

\bigskip


\section{Preliminaries}\label{S:prelim}

If $\Omega\subset\C^n$ is a domain, let $\co(\Omega)$ denote the holomorphic functions on $\Omega$. The standard $L^2$ inner product will be denoted

\begin{equation}\label{D:innerProduct}
\left\langle f,g\right\rangle =\int_\Omega f\cdot \bar g\, dV
\end{equation}
where $dV$ denotes Lebesgue measure on $\C^n$. $L^2(\Omega)$ denotes the measurable functions $f$ such that $\langle f,f\rangle =\|f\|^2 <\infty$ and $A^2(\Omega)=\co(\Omega)\cap L^2(\Omega)$.

The Bergman projection associated to $\Omega$ will be written $\mathbf{B}_\Omega$, or $\mathbf{B}$ if $\Omega$ is clear, and is the orthogonal projection operator
$\mathbf{B}:L^2(\Omega)\longrightarrow A^2(\Omega)$.
It is elementary that $\mathbf{B}$ is self-adjoint with respect to the inner product \eqref{D:innerProduct}. The Schwarz kernel of $\mathbf{B}_\Omega$ is the Bergman kernel, denoted $\B_\Omega(z,w)$, which
satisfies

\begin{equation*}
\mathbf{B}_\Omega f(z)=\int_\Omega \B_\Omega(z,w) f(w)\, dV(w),\qquad f\in L^2(\Omega).
\end{equation*}

For $p>0$, denote

\begin{equation*}
L^p(\Omega) =\left\{f\, :\,\,\left(\int_\Omega |f|^p\, dV\right)^{\frac 1p}:= \|f\|_p <\infty\right\}
\end{equation*}
and let $A^p(\Omega)=\co(\Omega)\cap L^p(\Omega)$.

Finally, we use the following notation to simplify writing various inequalities. If $A$ and $B$ are functions depending on several variables, write $A\lesssim B$ to signify that there exists a 
constant $K>0$, independent of relevant variables, such that $A\leq K\cdot B$. The independence of which variables will be clear in context. Also write $A\approx B$ to mean that
$A\lesssim B\lesssim A$.
\bigskip

\subsection{The Bergman kernel on fat Hartogs triangles}

The Bergman kernel for the fat Hartogs triangle of exponent $k \in \Z^+$ is explicitly computed in \cite{Edh15}, and its formula underlies the work done in this paper.  Let $\B_k$ denote the Bergman kernel on $\Omega_k$, and for simplicity write $s := z_{1}\bar{w}_{1}, \, t := z_{2}\bar{w}_{2}$.  Then

\begin{equation*}
\B_{k}(z,w) = \frac{p_{k}(s)t^{2}+q_{k}(s)t + s^{k}p_{k}(s)}{k\pi^{2}(1-t)^{2}(t -s^{k})^{2}},
\end{equation*}
where $p_k$ and $q_k$ are polynomials given by

\begin{equation*}
p_{k}(s) = \sum_{n=1}^{k-1} n(k-n)s^{n-1},\ q_{k}(s) = \sum_{n=1}^{k}(n^2 + (k-n)^2 s^k)s^{n-1}.
\end{equation*}
Since $\Omega_k$ is a bounded domain where $|s|^k < |t| < 1$, it follows that $\B_k$ satisfies the crucial estimate

\begin{equation}\label{E:basicEstimate}
|\B_{k}(z,w)| \lesssim \frac{|t|}{|1-t|^2|t-s^k|^2}.
\end{equation}
\bigskip

\subsection{Boundedness and Dual Spaces} The fact that the Bergman projection is self-adjoint, together with H\" older's inequality, show that the range of $p$ for which $\mathbf{B}$ is $L^p$ bounded in symmetric about $L^2$:

\begin{lemma}\label{L:symmetricLp}  Let $\Omega$ be a bounded domain and let $p > 1$.  If $\mathbf{B}$ maps $L^p(\Omega)$ to $A^p(\Omega)$ boundedly, then it also maps $L^q(\Omega)$ to $A^q(\Omega)$ boundedly, where $\frac 1p+ \frac 1q =1$.
\end{lemma}

\begin{proof}
This follows from a standard duality argument. Let $f \in L^q(\Omega)$.  Then

\begin{align*}
\left\|\mathbf{B}f\right\|_q &= \sup_{\norm{g}_p = 1} \left|\left\langle \mathbf{B}f, g \right\rangle\right| = \sup_{\norm{g}_p = 1} \left|\left\langle f, \mathbf{B}g \right\rangle\right| \\
	&\leq \sup_{\norm{g}_p = 1} \left( \norm{f}_q \norm{\mathbf{B}g}_p \right)
	\lesssim \norm{f}_q.
\end{align*}

\end{proof}
\bigskip

\subsection{A version of Schur's Lemma} The term on the right-hand side of the estimate \eqref{E:basicEstimate} is not integrable on $\Omega_k$, which prevents the classical Young's test for $L^p$ boundedness to be used in our case. See, e.g. \cite{FollandPDE2} Theorem 0.10 for Young's test and \cite{CucMcN06} for some generalizations of it. However, the bounding term just barely fails to be integrable; a computation shows that its $L^1$ norm (in $w$) is essentially bounded by $-\log\delta(z)$ near $0$, where $\delta(z)$ is the distance of $z$ to $b\Omega_k$.  In this situation, a substitute for Young's test, due to Schur, can be used to give $L^p$ boundedness. See, e.g. \cite{McNSte94}. We shall need a modification of Schur's lemma, which relates the range of exponents $\epsilon$ of the auxiliary function and the mapping classes $L^p$.

\begin{lemma}[Schur's Lemma]\label{L:SchursLemma}  Let $\Omega \subset \C^n$ be a domain, $K$ be an a.e. positive, measurable function on $\Omega \times \Omega$, and $\ck$ be the integral operator with kernel $K$. Suppose there exists a positive auxiliary function $h$ on $\Omega$, and numbers $0 < a < b$ such that for all $\epsilon \in [a,b)$, the following estimates hold:

\begin{align*}
\ck(h^{-\epsilon})(z) := \int_{\Omega} K(z,w) h(w)^{-\epsilon} \,dV(w)  \lesssim h(z)^{-\epsilon} \\
\ck(h^{-\epsilon})(w) := \int_{\Omega} K(z,w) h(z)^{-\epsilon} \,dV(z)  \lesssim h(w)^{-\epsilon}. \\
\end{align*}
Then $\ck$ is a bounded operator on $L^p(\Omega)$, for all $p \in (\frac{a+b}{b},\frac{a+b}{a})$.

\begin{proof}
Let $p$ and $q$ be finite conjugate exponents, and temporarily say that $q \le 2 \le p$.  Let $g \in L^p(\Omega)$, and $s \in [a,b)$ be a yet to be specified number.  Then H\"older's inequality gives
\begin{align*}
|\ck(g)(z)|^p 
	&\le \bigg( \int_{\Omega} K(z,w) |g(w)|^p h(w)^{sp/q} \, dV(w) \bigg) \bigg( \int_{\Omega} K(z,w) h(w)^{-s} \, dV(w) \bigg)^{p/q} \\
	&\lesssim \bigg( \int_{\Omega} K(z,w) |g(w)|^p h(w)^{sp/q} \, dV(w) \bigg) h(z)^{-sp/q}
\end{align*}
Now integrate in the $z$ variable to find
\begin{align*}
\int_{\Omega} |\ck(g)(z)|^p \, dV(z)
	&\lesssim  \int_{\Omega} \int_{\Omega} K(z,w) |g(w)|^p h(w)^{sp/q} h(z)^{-sp/q} \, dV(w) \, dV(z)  \\
	&= \int_{\Omega} \bigg( \int_{\Omega} K(z,w) h(z)^{-sp/q} \, dV(z) \bigg) h(w)^{sp/q} |g(w)|^p \,dV(w) \\
	&\lesssim \int_{\Omega} |g(w)|^p \,dV(w)
\end{align*}
where the last line is guaranteed to hold if and only if we can choose $s \in [a,b)$ so that $\frac{sp}{q} \in [a,b)$ as well.  This is equivalent to saying $\frac{ap}{q} = a(p-1) < b$, i.e. $p < \frac{a+b}{a}$.  Now let $g \in L^q(\Omega)$ and repeat the estimations above with $p$ and $q$ switched.  This time, we are able to conclude that $\norm{\ck(g)}_q \lesssim \norm{g}_q$, provided that we can choose $s \in [a,b)$ so that $\frac{sq}{p} \in [a,b)$ as well.  This is only possible if $\frac{bq}{p} = b(q-1) > a$, i.e. $q > \frac{a+b}{b}$.
We finish by removing the restriction that $p \ge 2$.

\end{proof}
\end{lemma}

\begin{remark}\label{R:Schur} In practice, the auxiliary function $h(w)$ will vanish on the set where the kernel $K(z,w)$ is singular (or large, as a function of $z$). Thus the product $K(z,w) h(w)^{-\epsilon}$ is {\it more} singular than $K(z,w)$, but its $L^1$ norm in $w$ is now algebraic rather than logarithmic in $\delta(z)$. See Remark \ref{R:Asymptotics}.
\end{remark}

\bigskip


\section{$L^p$ mapping properties}\label{S:LpMap}

\subsection{Non-$L^p$ boundedness}

\begin{theorem}\label{T:nonLp}
Let $p \ge 1$ be any number outside of the interval $(\frac{2k+2}{k+2}, \frac{2k+2}{k}).$  Then the Bergman projection $\mathbf{B}_k$ is not a bounded operator on $L^p(\Omega_k)$. 

\begin{proof}
In \cite{Edh15} it is shown that the set of monomials of the form $z^{\alpha}$ for the multi-indices $\alpha \in \ca_k = \{(\alpha_1, \alpha_2): \alpha_1 \ge 0 ,\, \alpha_{1} + k(\alpha_{2} +1) > -1   \}$ form an orthogonal basis for $A^2(\Omega_k)$.  Therefore, we can write the Bergman kernel as 

\begin{equation*}
\B_k(z,w) = \sum_{\alpha \in \ca_k} \frac{z^{\alpha}\bar{w}^{\alpha}}{c_{k,\alpha}^2},
\end{equation*}
where $c_{k,\alpha}$ is a normalizing constant.  Now set $f(z) := \bar{z}_2$; this is a bounded function on ${\Omega_k}$, so $f \in L^p({\Omega_k})$ for all $p \geq 1$.  It follows that

\begin{align*}
\mathbf{B}_k(f)(z) &= \int_{\Omega_k}\B_k(z,w)f(w)\,dV(w) = \int_{\Omega_k}\sum_{\alpha \in \ca_k} \frac{z^{\alpha}\bar{w}^{\alpha}}{c_{k,\alpha}^2}\cdot f(w)\,dV(w)\\
	&=\sum_{\alpha \in \ca_k} \frac{z^{\alpha}}{c_{k,\alpha}^2} \int_{\Omega_k} \bar{w}_1^{\alpha_1}\bar{w}_2^{\alpha_2+1} \, dV(w)\\
	&=\sum_{\alpha \in \ca_k} \frac{z^{\alpha}}{c_{k,\alpha}^2} \int_{\Omega_k} r_1^{\alpha_1} e^{-i{\theta_1}{\alpha_1}}r_2^{\alpha_2 +1} e^{-i{\theta_2}({\alpha_2}+1)}r_1 r_2 \,dr \,d\theta \\
	&=\sum_{\alpha \in \ca_k} \frac{z^{\alpha}}{c_{k,\alpha}^2}\left( \int_0^{2\pi} e^{-i{\theta_1}{\alpha_1}} \, d{\theta_1} \right) \left( \int_0^{2\pi} e^{-i{\theta_2}{(\alpha_2 + 1})} \, d{\theta_2} \right)\left( \int_{\omega_k} r_1^{\alpha_1 + 1}r_2^{\alpha_2 + 2} \, dr \right)\\
	&= \frac{C}{z_2},
\end{align*}
where $\omega_k = \{(r_1,r_2) \in \R^2: r_1 \ge 0, r_1^k < r_2 <1 \}$.  Here $C$ is a constant, and we used the fact that $\theta$ integrals vanish unless $\alpha_1 = 0, \alpha_2 = -1$. Thus,

\begin{align*}
\norm{\mathbf{B}_k(f)}_{p}^p &\approx \int_{\Omega_k}\frac{1}{|z_2|^p} = \int_{\Omega_k}\frac{1}{r_2^p}r_1 r_2 \, dr = 4 \pi^2 \int_{\omega_k}r_1 r_2^{1-p} \, dr\\
	&= 4 \pi^2 \int_0^1 r_2^{1-p}\int_0^{r_2^{1/k}} r_1 \, dr_1 \, dr_2\\
	&= 2 \pi^2 \int_0^1 r_2^{1-p+\frac{2}{k}} \, dr_2.
\end{align*}

This integral diverges when $p \ge \frac{2k+2}{k}$, so $\mathbf{B}_k(f)\notin L^p(\Omega_k)$ for this range of $p$.
Lemma \ref{L:symmetricLp} says that $\mathbf{B}_k$ also fails to be a bounded operator on $L^p(\Omega_k)$ when $p \in (1,\frac{2k+2}{k+2}]$.

\end{proof}
\end{theorem}
\bigskip

\subsection{An estimate related to $\B_D$} As prelude for constructing the auxiliary function needed in Schur's lemma, we first establish an
integral estimate on the Bergman kernel on the unit disc in $\C$, $\B_D(z,w)$. Note the piece of the integrand below $(1-|w|^2)$ corresponding to
the defining function for $D$. The factor $|w|^{-\beta}$ in the integrand is needed because of a Jacobian factor in the proof of Lemma \ref{L:calculus2} below.

\begin{lemma} \label{L:calculus1}
For $z \in D$, the unit disk in $\C$, let $\epsilon \in (0,1)$ and $\beta \in [0,2)$.  Then
$$\ci_{\epsilon,\beta}(z) := \int_D \frac{(1 - |w|^2)^{-\epsilon}}{|1 - z\bar{w}|^2}|w|^{-\beta} \, dV(w) \lesssim (1 - |z|^2)^{-\epsilon}.$$
\end{lemma}

\begin{proof}
This result has been re-discovered many times, see for instance \cite{RudinFunctiontheory}, \cite{ZhuBergmanbook}, \cite{Chen14},
with a proof based on power series. We give a different proof here.
  
First consider the case when $|z| \le \frac{1}{2}$.  Then $|1 - z\bar{w}| \ge 1 - |z\bar{w}| \ge \frac{1}{2}$, and so 

\begin{align*}
\ci_{\epsilon,\beta}(z) &\le 4 \int_{D} (1 - |w|^2)^{-\epsilon}|w|^{-\beta} \, dV(w) \\
	&= 4\pi \int_0^1 (1-u)^{-\epsilon} u^{-\beta/2} \, du < \infty.
\end{align*}
Since this integral is bounded by a constant independent of $z$, the desired estimate holds.

Now let $|z| > \frac{1}{2}$, and for the rest of this proof let $c := \frac{1}{2|z|}$.  Now split up the integral into two pieces, called $I_1$ and $I_2$:
\begin{align*}
\ci_{\epsilon,\beta}(z) &= \int_{|w| \le c} \frac{(1 - |w|^2)^{-\epsilon}}{|1 - z\bar{w}|^2}|w|^{-\beta} \, dV(w) + \int_{|w| > c} \frac{(1 - |w|^2)^{-\epsilon}}{|1 - z\bar{w}|^2}|w|^{-\beta} \, dV(w) \\
&:= I_1 + I_2
\end{align*}

For $I_1$, notice that $|1 - z\bar{w}| \ge 1 - |z\bar{w}| \ge \frac{1}{2}$, and thus

\begin{align*}
I_1 \le 4 \int_{|w| \le c} (1 - |w|^2)^{-\epsilon}|w|^{-\beta} \, dV(w) 
	< 4 \int_{D} (1 - |w|^2)^{-\epsilon}|w|^{-\beta} \, dV(w) < \infty,
\end{align*}
and so $I_1$ satisfies the required estimate.  

It remains to be seen that $I_2$ does too.  Since $\frac{1}{2} < |z| < 1$, it must be that $\frac{1}{2} < c < 1$, and consequently,  $\frac{1}{2} < |w| < 1$ throughout all of $I_2$.  So for $\beta \in [0,2)$, we have that $1 \le |w|^{-\beta} < 4$.  This lets us make the estimate
\[
I_2 \le 4 \int_{|w| > c} \frac{(1 - |w|^2)^{-\epsilon}}{|1 - z\bar{w}|^2} \, dV(w).
\]
Now,
\begin{align*}
\int_{|w| > c} \frac{(1 - |w|^2)^{-\epsilon}}{|1 - z\bar{w}|^2} \, dV(w) 
	&= \int_c^1 r(1-r^2)^{-\epsilon} \bigg[ \int_0^{2\pi} \frac{d\theta}{1 - 2r|z|\cos{\theta}+ r^2|z|^2} \bigg] dr \\
	&= 2 \int_c^1 r(1-r^2)^{-\epsilon} \bigg[ \int_0^{\pi} \frac{d\theta}{1 - 2r|z|\cos{\theta}+ r^2|z|^2}\bigg] dr. \\
\end{align*}
Now, estimate the integral in brackets by using the fact that $\sin\phi \ge \frac{2\phi}{\pi}$ for $\phi \in [0,\frac{\pi}{2}]$, along with the half-angle formula.
\begin{align*}
\int_0^{\pi} \frac{d\theta}{1 - 2r|z|\cos{\theta}+ r^2|z|^2}
	&< \int_0^{\pi} \frac{d\theta}{(1 - r|z|)^2 + \frac{2\theta^2}{\pi^2}}\\
	&= (1 - r|z|)^{-2} \int_0^{\pi} \frac{d\theta}{1 + \frac{2}{\pi^2}\big(\frac{\theta}{1 - r|z|}\big)^2}\\
	&< (1 - r|z|)^{-1} \int_0^{\infty} \frac{du}{1 + \frac{2}{\pi^2}u^2}\\
	&\lesssim (1 - r|z|)^{-1}.
\end{align*}

Therefore, we have that
\begin{equation*}
I_2 \lesssim \int_c^1 r(1-r^2)^{-\epsilon} (1 - r|z|)^{-1} \, dr
\end{equation*}
For reasons that will soon be apparent, we want to have $|z|$ in the bounds of our integral.  Unfortunately, $|z|$ and $c = \frac{1}{2|z|}$ aren't comparable quantities.  So we make one more trivial over-estimate on $I_2$ and split this into two pieces, called $J_1$ and $J_2$.
\begin{align*}
I_2 &\lesssim \int_0^{|z|} r(1-r^2)^{-\epsilon} (1 - r|z|)^{-1} \, dr + \int_{|z|}^1 r(1-r^2)^{-\epsilon} (1 - r|z|)^{-1} \, dr \\
	&:= J_1 + J_2
\end{align*}

Now for $J_1$, since $0 \le r \le |z|$, we can say that
\begin{align*}
J_1 < \int_0^{|z|} r(1-r|z|)^{-\epsilon-1} \, dr &= -\frac{1}{|z|} \int_1^{1-|z|^2} u^{-\epsilon-1} \, du \\
	&\lesssim (1 - |z|^2)^{-\epsilon}.
\end{align*}

Finally, for $J_2$, since we know that $r \le 1$,
\begin{align*}
J_2 < (1 - |z|)^{-1} \int_{|z|}^1 r(1 - r^2)^{-\epsilon} \, dr &\lesssim (1 - |z|)^{-1} (1 - |z|^2)^{1 - \epsilon} \\
	&\lesssim (1 - |z|^2)^{-\epsilon}.
\end{align*}
This means that $I_2 \lesssim (1 - |z|^2)^{-\epsilon}$.  Therefore, $\ci_{\epsilon,\beta}(z)$ satisfies the desired estimate for all $z \in D$.

\end{proof} 

\begin{remark}\label{R:Asymptotics}
Now that we've done an explicit computation, we can illuminate Remark \ref{R:Schur}. Let $\delta(z)=
\delta_D(z)$ denote the distance of $z$ to $bD$, and $\left|\mathbf{B}_D\right|$ be the integral operator with kernel $\left|\B_D\right| \approx |1 - z\bar{w}|^{-2}$. With a slight modification to the proof of Lemma \ref{L:calculus1}, we see that for $z$ sufficiently close to $bD$,

\begin{itemize}
\item[(i)] $\left|\mathbf{B}_D\right|\left[1\right](z)\lesssim -\log\delta(z)$ and
\item[(ii)] $\left|\mathbf{B}_D\right|\left[\delta^{-\epsilon}\right](z)\lesssim\delta^{-\epsilon}(z)$.
\end{itemize}
Further, the inclusion of an integrable singular function with singularity away from $bD$ doesn't change these asymptotic estimates:  
\begin{itemize}
\item[(iii)] $\left|\mathbf{B}_D\right|\left[\frac 1{|w|^\beta}\right](z)\lesssim -\log\delta(z), \quad \beta\in [0,2)$,   and
\item[(iv)] $\left|\mathbf{B}_D\right|\left[\frac {\delta(w)^{-\epsilon}}{|w|^\beta}\right](z)\lesssim\delta^{-\epsilon}(z), \quad \beta\in [0,2)$.
\end{itemize}
These inequalities hold in much greater generality, i.e., for many $\Omega\subset\C^n$ and $\mathbf{B}_\Omega$ besides $\Omega =D$.
\end{remark}

\bigskip

\subsection{The auxiliary function for Schur's lemma} An auxiliary function adapted to the domain $\Omega_k$ for which we can apply Schur's lemma is now constructed. The function $h$ in the following Lemma essentially measures the distance of $z\in\Omega_k$ to the boundary.

\begin{lemma}\label{L:calculus2} 
Let $k \in \Z^+$,  $h(z) := (|z_2|^2 - |z_1|^{2k})(1 - |z_2|^2),$ and $\B_k(z,w)$ denote the Bergman kernel on $\Omega_k$.  Then for all $\epsilon \in \big[\frac{1}{2},\frac{k+2}{2k}\big)$ and any $z \in \Omega_k$, we have that 

\begin{equation*}
|\mathbf{B}_k|(h^{-\epsilon})(z) := \int_{\Omega_k}|\B_k(z,w)|h^{-\epsilon}(w)\,dV(w) \lesssim  h^{-\epsilon}(z).
\end{equation*}
\end{lemma}

\begin{proof} From estimate \eqref{E:basicEstimate}, we see that

\begin{align*}
|\mathbf{B}_k|(h^{-\epsilon})(z) &\lesssim \int_{\Omega_k} \frac{|z_{2}\bar{w}_{2}| (|w_2|^2 - |w_1|^{2k})^{-\epsilon}(1 - |w_2|^2)^{-\epsilon}}{|1-z_{2}\bar{w}_{2}|^2|z_{2}\bar{w}_{2}-z_{1}^k\bar{w}_{1}^k|^2}\,dV(w)\\
	&= \int_{D^*} \frac{|z_{2}\bar{w}_{2}| (1 - |w_2|^2)^{-\epsilon}}{|1-z_{2}\bar{w}_{2}|^2} \bigg[ \int_{W} \frac{(|w_2|^2 - |w_1|^{2k})^{-\epsilon}}{|z_{2}\bar{w}_{2}-z_{1}^k\bar{w}_{1}^k|^2} \, dV(w_1)\bigg] dV(w_2)
\end{align*}
where the outside integral is the over region $D^* := \{w_2:0 < |w_2| < 1\}$ and the integral in brackets is over the region $W := \{w_1 : |w_1| < |w_2|^{1/k}\}$, where $w_2$ is considered to be fixed.  Denote the integral in brackets by $A$ and focus on this.

\begin{align*}
A = \frac{1}{|z_2|^2 |w_2|^{2 + 2\epsilon}} \int_{W} \bigg(1 - \bigg|\frac{w_1^k}{w_2}\bigg|^2 \bigg)^{-\epsilon} \bigg| 1 - \frac{z_1^k\bar{w}_1^k}{z_2\bar{w}_2}\bigg|^{-2}  dV(w_1).
\end{align*}

Make the substitution $u = \frac{w_1^k}{w_2}$.  This transformation sends $W$ to $k$ copies of $D$, the unit disc in the $u$-plane.  Lemma \ref{L:calculus1} now yields

\begin{align*}
A &= \frac{|w_2|^{2/k - 2 - 2\epsilon}}{k |z_2|^2} \int_{D} \frac{(1 - |u|^2)^{-\epsilon}}{|1 - z_1^k z_2^{-1} \bar{u}|^2} \cdot |u|^{2/k -2} \, dV(u) \\
	&\lesssim \frac{|w_2|^{2/k -2 -2\epsilon}}{|z_2|^2} \bigg(1 - \bigg| \frac{z_1^k}{z_2}\bigg|^2 \bigg)^{-\epsilon} \\
	&=  \frac{|w_2|^{2/k -2 -2\epsilon}}{|z_2|^{2-2\epsilon}} (|z_2|^2 - |z_1|^{2k})^{-\epsilon}.
\end{align*}

This means that
\begin{align*}
|\mathbf{B}_k|(h^{-\epsilon})(z) &\lesssim |z_2|^{2\epsilon - 1} (|z_2|^2 - |z_1|^{2k})^{-\epsilon} \int_{D^*} \frac{(1 - |w_2|^2)^{-\epsilon}}{|1-z_{2}\bar{w}_{2}|^2} |w_2|^{2/k - 1 - 2\epsilon} \, dV(w_2)\\
	&\lesssim |z_2|^{2\epsilon - 1} (|z_2|^2 - |z_1|^{2k})^{-\epsilon} (1 - |z_2|^2)^{-\epsilon}
\end{align*}
where the last line holds if and only if the conditions on Lemma \ref{L:calculus1} hold.  This happens when $\frac{2}{k} - 1 - 2\epsilon > -2$, i.e. $\frac{k+2}{2k} > \epsilon$.  Finally, 
\[
|\mathbf{B}_k|(h^{-\epsilon})(z) \lesssim |z_2|^{2\epsilon - 1} (|z_2|^2 - |z_1|^{2k})^{-\epsilon} (1 - |z_2|^2)^{-\epsilon} \le h(z)^{-\epsilon}
\]
for all $z \in \Omega_k$ whenever $\epsilon \ge \frac{1}{2}$.  This completes the proof.
\end{proof}

\bigskip

\subsection{Proof of Theorem \ref{T:main}}  With the tools above in hand, the proof of our main result is easy to conclude:

\begin{proof}
Combining Lemma \ref{L:calculus2} and Schur's lemma (Lemma \ref{L:SchursLemma}) yields that the operator $|\mathbf{B}_k|$ is bounded from $L^p(\Omega_k)$ to $L^p(\Omega_k)$ for $p \in (\frac{2k+2}{k+2}, \frac{2k+2}{k})$.  A fortiori, $\mathbf{B}_k$ is bounded from $L^p(\Omega_k)$ to $A^p(\Omega_k)$ for $p$ in the same range.  Notice that because of the conjugate symmetry of $\B_k(z,w)$, it is sufficient to establish just one of the estimates to apply Lemma \ref{L:SchursLemma}.  

On the other hand, Theorem \ref{T:nonLp} shows that $p \in (\frac{2k+2}{k+2}, \frac{2k+2}{k})$ are the only values for which $\mathbf{B}_k$ is bounded on $L^p(\Omega_k)$. This completes the proof.
\end{proof}
\medskip

The exact same proof  yields a more general version of half of Theorem \ref{T:main}:

\begin{corollary}\label{C:extendedLp} Let $p \in (\frac{2k+2}{k+2},\frac{2k+2}{k})$. Using the previous notation, if $K(z,w)$ is any kernel on $\Omega_k$ satisfying

\begin{equation*}
|K(z,w)| \lesssim \frac{|t|}{|1-t|^2|t-s^k|^2},
\end{equation*}
then the operator 
$$\ck f(z) :=\int_{\Omega_k} \left|K(z,w)\right| f(w)\, dV(w)$$ 
maps $L^p\left(\Omega_k\right)$ to $L^p\left(\Omega_k\right)$ boundedly.

\end{corollary}
\bigskip

\subsection{Dual space of $A^p$} The $L^p$ boundedness of $\mathbf{B}$ in Theorem \ref{T:main} yields a simple characterization of the dual space of $A^p\left(\Omega_k\right)$, for $p \in (\frac{2k+2}{k+2},\frac{2k+2}{k})$. Namely, the graded Bergman spaces $A^p\left(\Omega_k\right)$ are H\" older-dual, precisely like the ordinary $L^p$ spaces, for these $p$.

Let $A^p(\Omega)^*$ denote the bounded linear functionals on $A^p(\Omega)$ and, for $\ell\in A^p(\Omega)^*$, let $\|\ell\|_{op}=\sup_{\|f\|_p=1} \left|\ell(f)\right|$ denote the usual operator
norm.

\begin{theorem}\label{T:dual space} If $\Omega_k$ denotes the domain  \eqref{D:fatHartogs}, then for $p \in (\frac{2k+2}{k+2},\frac{2k+2}{k})$ it holds that

\begin{equation*}
A^p\left(\Omega_k\right)^*\cong A^q\left(\Omega_k\right), \qquad\text{where}\,\,\,\,\,\frac 1p +\frac 1q =1.
\end{equation*}
Moreover, for each $\psi\in A^p\left(\Omega_k\right)^*$, there exists a unique $h_\psi\in A^q\left(\Omega_k\right)$ such that

\begin{equation}\label{E:rep1}
\psi(f)=\left\langle f, {h_\psi}\right\rangle
\end{equation}
with $\|\psi\|_{op}\approx \left\|h_\psi\right\|_q$.

\end{theorem}

\begin{proof}

If $h\in A^q\left(\Omega_k\right)$, H\"older's inequality implies that the linear functional

$$f \overset\ell\longrightarrow\int_{\Omega_k} f\cdot\bar h\, dm$$
is bounded and $\|\ell\|_{op}\leq \|h\|_q$.

Conversely, if $\psi\in A^p\left(\Omega_k\right)^*$, the Hahn-Banach theorem extends $\psi$ to a functional $\Psi\in L^p\left(\Omega_k\right)^*$ with $\|\Psi\|_{op}=\|\psi\|_{op}$.
The Riesz representation theorem gives a $g_\psi\in L^q(\Omega_k)$ such that

$$\Psi(f)=\left\langle f,{g_\psi}\right\rangle,\qquad f\in L^p(\Omega_k)$$
and $\|g_{\psi}\|_q=\|\Psi\|_{op}$.

The idea of the proof is very simple. If $g_\psi$ also belongs to $L^2(\Omega_k)$, set $h_\psi =\mathbf{B}g_\psi$. Then for $f\in A^p(\Omega_k)\cap A^2(\Omega_k)$,
the self-adjointness of $\mathbf{B}$  gives

\begin{align*}
\psi(f)&=\Psi(f)=\left\langle f,g_\psi\right\rangle \\ &=\left\langle \mathbf{B}f,g_\psi\right\rangle \\
&= \left\langle f,\mathbf{B}g_\psi\right\rangle = \left\langle f,h_\psi\right\rangle,
\end{align*}
thus $h_\psi$ represents the functional $\psi$ al\' a \eqref{E:rep1}. The issue is that both $L^p(\Omega_k), L^q(\Omega_k)\subset L^2(\Omega_k)$ only when $p=q=2$, so 
$g_\psi$ and a general $f\in A^p(\Omega_k)$ are not simultaneously in the domain of $\mathbf{B}$.

However, a slight modification of this basic idea yields the result. Suppose first that $p\in\left[2,\frac{2k+2}{k}\right)$ and let $f\in A^p\left(\Omega_k\right)$. Then
$f\in A^2\left(\Omega_k\right)$ since $\Omega_k$ is bounded.  As $g_\psi\in L^q$ and $q$ is in the allowable range, Theorem \ref{T:main} says that $h_\psi:=\mathbf{B}g_\psi\in A^q\left(\Omega_k\right)$,
and consequently the inner product $\left(f,h_\psi\right)$ is well-defined.
Since $\mathbf{B}f=f$, we have

\begin{align}\label{E:rep2}
\psi(f)&=\int_{\Omega_k}f\cdot\overline{g_\psi}\notag \\
&=\int_{\Omega_k}\left(\int_{\Omega_k}\B_k(z,w)f(w)\, dV(w)\right)\cdot\overline{g_\psi}(z)\, dV(z).
\end{align}
We now want to apply Fubini's theorem to the last integral. This is justified since Corollary \ref{C:extendedLp} says the
operator

\begin{equation*}
g_\psi\longrightarrow \int_{\Omega_k}\left|\B_k(z,w)\right|\, \left|g_\psi(z)\right|\, dV(z)
\end{equation*}
maps $L^q\left(\Omega_k\right)$ to $L^q\left(\Omega_k\right)$ boundedly. It follows that

\begin{align*}
\eqref{E:rep2}&= \int_{\Omega_k}f(w)\left(\int_{\Omega_k}\B_k(z,w)\overline{g_\psi}(z)\, dV(z)\right)\, dV(w) \\
&=\int_{\Omega_k}f(w)\overline{h_\psi}(w)\, dV(w) =\left\langle f,h_\psi\right\rangle
\end{align*}
as desired.

If $p\in\left(\frac{2k+2}{k+2},2\right]$ instead, run the argument above backwards, starting with

$$\left\langle f,h_\psi\right\rangle =\int_{\Omega_k} f(w)\,\,\overline{\int_{\Omega_k}\B_k(w,z)g_\psi(z)\, dV(z)}\,\, dV(w)$$
and see that \eqref{E:rep1} holds in this case too.

\end{proof}

\section{Concluding Remarks}

1. Theorem \ref{T:main} holds for $k \in \Z^+$.  If $\Omega_r$ is the fat Hartogs triangle of exponent $r \in (1,\infty)$, it is reasonable to expect a restricted range of $p$ for which the Bergman projection $\mathbf{B}_r$ is a bounded operator on $L^p(\Omega_r)$.  In fact, Theorem \ref{T:nonLp} is still valid for non-integer exponents, which rules out boundedness for any $p \notin (\frac{2r+2}{r+2}, \frac{2r+2}{r})$.  It would be interesting to deduce an analog of estimate (\ref{E:basicEstimate}), which could help us see which exponents work. 

\medskip

2.  For $\Omega_k$, $k \in \Z^+$, the proof of Theorem \ref{T:nonLp} shows a single anti-holomorphic monomial establishes the entire range of $p$ for which the Bergman projection is not a bounded operator.  Thus the $L^p$ behavior of $\mathbf{B}_k$ on the anti-holomorphic subspace $\overline{A^p(\Omega_k)} \subset L^p(\Omega_k)$ determines the behavior of $\mathbf{B}_k$ on the full space.  What conditions are needed on a domain $\Omega$ for this surprising observation to occur?  It's worth mentioning that if $\Omega$ is a Reinhardt domain containing the origin, the Bergman projection of a non-constant anti-holomorphic monomial yields $0$.

\medskip

3. Clearly the singularity of $b\Omega_k$ at $0$ limits the $L^p$ mapping behavior of its Bergman projection, but the exact nature of this type of singularity is unexplored.  How does $\mathbf{B}$ change if we replace $|z_1|^k$ in definition (\ref{D:fatHartogs}) with a more general $\psi\left(\left|z_1\right|\right)$? What happens on non-Reinhardt domains with this type of boundary singularity?  Can $\Omega_k$ be used as a model domain for a class of invariantly defined domains in $\C^2$ with singular boundary?

\medskip

4. It would be interesting to investigate higher dimensional domains with analogous boundary singularities.  Chen, \cite{Chen14}, introduces a class of domains in $\C^n$ which generalize the classical Hartogs triangle.  Domains of the form  $$ \{(z_1,\cdots,z_n): |z_1|^{k_1} < |z_2|^{k_2} < \cdots < |z_{n-1}|^{k_{n-1}} < |z_n| < 1 \}$$
generalize the Hartogs triangle differently than Chen does, and are the natural extension of the fat Hartogs triangles to higher dimensions.  These domains would provide an intriguing class of examples from which the interplay between geometry and function theory might be further understood.
\medskip

5. We have not addressed endpoint results in this paper, but it is natural to ask: where does $\mathbf{B}_k$ map $L^{\frac{2k+2}{k+2}}\left(\Omega_k\right)$ and $L^{\frac{2k+2}{k}}\left(\Omega_k\right)$? For both endpoints,
it seems likely that $\mathbf{B}_k$ maps $L^r\left(\Omega_k\right)$ into weak $L^r$, analogous to the weak-type $(1,1)$ mapping behavior of $\mathbf{B}_\Omega$ known to hold on classes of (smoothly bounded) finite type domains, see \cite{McNeal94b}. For the upper endpoint, it may be possible to say more. Recall that for $D$, the unit disc in $\C$, $\mathbf{B}_D$ maps $L^\infty(D)$ {\it onto} the holomorphic functions satisfying $\sup_{z\in D} \left(1-|z|^2\right)|f'(z)| < \infty$, the classical Bloch space. It would be interesting to find a similar characterization of the range of $\mathbf{B}_k$ acting on $L^{\frac{2k+2}{k}}\left(\Omega_k\right)$, though the appropriate ``Bloch-like condition'' on $\Omega_k$ is unclear at this point, partly because $b\Omega_k$ is not smooth.

\bigskip
\bigskip

\bibliographystyle{acm}
\bibliography{EdhMcN}

\end{document}